\documentclass[12pt]{article}
\def\date{9 August 2016}
\usepackage{amsmath, amssymb, amsfonts, amsthm, color, soul}

\newtheorem{proposition}{Proposition}[section]
\newtheorem{lemma}[proposition]{Lemma}
\newtheorem{corollary}[proposition]{Corollary}
\newtheorem{theorem}[proposition]{Theorem}

\newtheorem{conjecture}[proposition]{Conjecture}

\newtheorem{question}[proposition]{Question}

\newtheorem{thm}[proposition]{Theorem}

{
\theoremstyle{definition}
\newtheorem{definition}[proposition]{Definition}

}

\newcommand{\mcal}{\mathcal}
\newcommand{\A}{\mcal{A}}

\newcommand{\G}{\mcal{G}}

\renewcommand{\P}{\mcal{P}}

\begin{document}
\font\smallrm=cmr8

\baselineskip=12pt
\phantom{a}\vskip .25in
\centerline{{\bf Bounded Diameter Arboricity}}
\vskip.4in
\centerline{{\bf Martin Merker}%
\footnote{\texttt{marmer@dtu.dk.} Supported by ERC Advanced Grant GRACOL, project number 320812.}}
\smallskip
\centerline{Department of Applied Mathematics and Computer Science}
\smallskip
\centerline{Technical University of Denmark}
\smallskip
\centerline{Lyngby, Denmark}
\bigskip
\centerline{and}
\bigskip
\centerline{{\bf Luke Postle}
\footnote{\texttt{lpostle@uwaterloo.ca.} Canada Research Chair in Graph Theory. Partially supported by NSERC under Discovery Grant No. 2014-06162, the Ontario Early Reseacher Awards program and the Canada Research Chairs program.}\textsuperscript{,}\footnote{\label{note3}Part of this research was carried out during a visit of the second author at the first author's institution.}} 
\smallskip
\centerline{Department of Combinatorics and Optimization}
\smallskip
\centerline{University of Waterloo}
\smallskip
\centerline{Waterloo, Ontario, Canada}

\vskip 1in \centerline{\bf ABSTRACT}
\bigskip

{
\parshape=1.0truein 5.5truein
\noindent
We introduce the notion of \emph{bounded diameter arboricity}. Specifically, the \emph{diameter-$d$ arboricity} of a graph is the minimum
number $k$ such that the edges of the graph can be partitioned into $k$ forests each of whose components has diameter at most $d$. A class of graphs has bounded diameter arboricity $k$ if there exists a natural number $d$ such that every graph in the class has diameter-$d$ arboricity at most $k$. 
We conjecture that the class of graphs with arboricity at most $k$ has bounded diameter arboricity at most $k+1$. We prove this conjecture for $k\in \{2,3\}$ by proving the stronger assertion that the union of a forest and a star forest can be partitioned into two forests of diameter at most 18. 
We use these results to characterize the bounded diameter arboricity for planar graphs of girth at least $g$ for all $g\ne 5$. As an application we show that every 6-edge-connected planar (multi)graph contains two edge-disjoint $\frac{18}{19}$-thin spanning trees. 
}

\vfill \baselineskip 11pt \noindent \date.
\vfil\eject
\baselineskip 18pt

\section{Introduction}

In this paper, all graphs are considered to be finite and simple unless stated otherwise. A \emph{decomposition} of a graph $G$ consists of edge-disjoint subgraphs whose union is $G$. The \emph{arboricity} $\Upsilon(G)$ of a graph $G$ is the minimum number $k$ such that $G$ can be decomposed into $k$ forests. An obvious necessary condition for a graph to have arboricity at most $k$ is that $|E(H)|\le k(|V(H)|-1)$ for every subgraph $H$ of $G$. In 1964, Nash-Williams~\cite{NashWilliams}  proved that this condition is also sufficient. Given Nash-Williams' result it is clear that arboricity is closely linked to the \emph{maximum average degree} of a graph, that is the maximum of the average degree of all subgraphs.

A natural question is to wonder what conditions are necessary to demand more structure from the forests in the partition. To date, there have been two prominent variants of arboricity that address this question. The first variant, introduced by Harary~\cite{Harary} in 1970 is \emph{linear arboricity}, denoted $\Upsilon_\ell (G)$, wherein the graph is decomposed into linear forests, where a \emph{linear forest} is the disjoint union of paths. Clearly, the linear arboricity is intimately connected to the maximum degree of a graph as $\Delta(G)/2 \le \Upsilon_\ell (G) \le \Delta(G)+1$. The \emph{linear arboricity conjecture} posed in~\cite{LAConj} states that $\Upsilon_\ell (G) \le \frac{\Delta(G)+1}{2}$. Using probabilistic methods, Alon~\cite{Alon} showed that $\Upsilon_\ell (G) = \Delta(G)/2 + O(\Delta(G) \log\log \Delta(G) / \log \Delta(G))$. 

The second variant is \emph{star arboricity}, denoted $\Upsilon_s(G)$, wherein the graph is decomposed into star forests, where a \emph{star forest} is the disjoint union of stars. Clearly, the star arboricity is at most twice the arboricity since every forest can be decomposed into two star forests. However, this is best possible due to a construction of Alon et al.~\cite{Star92}. Nevertheless for certain interesting graph classes the star arboricity can be lower. For example, Algor and Alon~\cite{Star1} showed that $\Upsilon_s(G) \le d/2 + O(d^{2/3}\log^{1/3} d)$ for $d$-regular graphs. Hakimi et al.~\cite{Star2} showed that $\Upsilon_s(G)$ is at most the acyclic chromatic number of $G$. In 1979, Borodin~\cite{Borodin} showed that planar graphs have acyclic chromatic at most 5 and thus star arboricity at most 5, which is best possible as shown by Algor and Alon~\cite{Star1}.

Yet both of these definitions seem unsatisfactory given the above results, linear arboricity for its relationship to the maximum degree and star arboricity for its lack of general improvement over the trivial bound. Therefore, we are motivated to find a new form of additional structure which could avoid these problems. However, requiring the components of the forest to have bounded size is infeasible given that a star has arboricity one. Thus, we introduce the following definition.

\begin{definition}
The \emph{diameter-$d$ arboricity} $\Upsilon_d(G)$ of a graph $G$ is the minimum number $k$ such that the edges of $G$ can be partitioned into $k$ forests each of whose components have diameter at most $d$. The \emph{bounded diameter arboricity} $\Upsilon_{bd}(\G)$ of a class of graphs $\G$ is the minimum number $k$ for which there exists a natural number $d$ such that every $G\in \G$ has diameter-$d$ arboricity at most $k$. 
\end{definition}

While we are interested in what diameters can be obtained, we are in general more interested in which graphs have any bound on the diameter. Such a notion though only makes sense when referring to graph classes, e.g. planar graphs or graphs of arboricity at most $k$. To that end let $\mathcal{A}_k$ denote the class of graphs with arboricity at most $k$. Clearly $\Upsilon_{bd}(\A_k) \le \Upsilon_2(\A_k) \le 2k$ since every forest can be partitioned into two star forests. Similarly $\Upsilon_{bd}(\mathcal{A}_k)$ is strictly greater than $k$. To see this note that a graph which is the union of $k$ spanning trees if decomposed into $k$ forests must necessarily be decomposed into $k$ spanning trees. Since there exists graphs of arbitrarily large diameter which are the union of $k$ spanning trees it follows that every such decomposition has a forest with a component of large diameter. But is it possible that by allowing a few more forests we can in fact obtain components of bounded diameter? We make the following very strong conjecture.

\begin{conjecture}\label{PlusOneConj}
The class of graphs with arboricity at most $k$ has bounded diameter arboricity $k+1$, i.e. $\Upsilon_{bd}(\mathcal{A}_k) = k+1$.
\end{conjecture} 

Our first main result is that we have verified this conjecture for $k=2$ and $k=3$.

\begin{thm}\label{23}
$\Upsilon_{bd}(\mathcal{A}_2) = 3$ and $\Upsilon_{bd}(\mathcal{A}_3) = 4$.
\end{thm}

This result can be used to improve the general upper bound $\Upsilon_{bd}(\mathcal{A}_k) \le 2k$. 

\begin{corollary}\label{FourThirds}
$\Upsilon_{bd}(\mathcal{A}_k) \le \lceil \frac{4}{3} k \rceil$.
\end{corollary}

\begin{proof}
By Theorem~\ref{23} there exists a natural number $d$ such that every graph $G$ with $\Upsilon (G) \leq 3$ satisfies $\Upsilon_d (G)\leq \Upsilon (G)+1$. If $\Upsilon (G)=k$, then $G$ can be written as the union of $\ell = \lceil \frac{k}{3}\rceil$ graphs $G_1,\ldots ,G_\ell$ with $\Upsilon (G_i)= 3$ for $i\in \{1,\ldots ,\ell -1\}$ and $\Upsilon (G_\ell ) = k-3(\ell -1)$. 
Now 
$$\Upsilon_d (G) \leq \Upsilon_d (G_1) + \ldots + \Upsilon_d (G_\ell ) \leq 4(\ell -1) + k - 3(\ell -1) +1 = \left\lceil \frac{4k}{3}\right\rceil$$
and thus $\Upsilon_{bd}(\mathcal{A}_k) \le \lceil \frac{4}{3} k \rceil$.
\end{proof}

To prove Theorem~\ref{23}, we in fact prove the following stronger theorem, which we prove in Section 2.

\begin{thm}\label{TreeStar}
If $G$ is the union of a forest and a star forest, then $\Upsilon_{18}(G)\leq 2$.
\end{thm}

\begin{corollary}\label{2}
If $\Upsilon(G)\le 2$, then $\Upsilon_{18}(G)\le 3$. If $\Upsilon(G)\le 3$, then $\Upsilon_{18}(G)\le 4$. In particular, Conjecture~\ref{PlusOneConj} holds for $k=2$ and $k=3$.
\end{corollary}
\begin{proof}
If $G$ is the union of two forests $F_1$ and $F_2$, then we decompose the edges of $F_2$ into two star forests $S_1$ and $S_2$ and apply Theorem~\ref{TreeStar} to the union of $F_1$ and $S_1$.

If $G$ is the union of three forests $F_1$, $F_2$ and $F_3$, then we decompose the edges of $F_3$ into two star forests $S_1$ and $S_2$. We apply Theorem~\ref{TreeStar} to the union of $F_1$ and $S_1$, and separately to the union of $F_2$ and $S_2$.
\end{proof}

In Section 3 we investigate the bounded diameter arboricity for planar graphs of a certain girth and show the existence of $\varepsilon$-thin spanning trees in highly edge-connected planar graphs.

\section{Forest plus Star Forest}

In this section we show that every simple graph $G$ which is the union of a forest and a star forest can be decomposed into two forests in which every tree has diameter at most 18. Note that our proof also works if we allow $G$ to be infinite.

An \textit{out tree} is a rooted tree in which every edge is oriented away from the root.
An \textit{out star forest} is a directed forest in which every component is a star and the edges of every star are oriented from the center to the leaves. If a star has size 1, then we arbitrarily choose one of the two vertices as the center and orient the edge away from it.

\begin{definition}\label{def-outing}
An \emph{outing} $G=(S,T)$ is the union of an out star forest $S$ and an out tree $T$. We let $C(S)$ denote the set of centers of the star forest $S$ and $L(S)$ denote the set of leaves of $S$. 
\end{definition}

Given an outing $G$, our goal is to construct a 2-edge-coloring of $G$ such that there are no monochromatic cycles and no long monochromatic paths. Notice that in an outing every vertex has indegree at most 2. The first important property of the coloring we construct is that every vertex has indegree at most 1 in each color. In such a coloring every monochromatic cycle is directed and every monochromatic path is the union of at most two directed paths which we call \emph{dipaths} for brevity.

Ideally we would like to start with an edge-coloring of $S$ in which every star is monochromatic and extend this coloring to all edges of $G$. Unfortunately, this additional constraint is too strong: If a monochromatic star has $d$ leaves which form a path in $T$, then coloring the edges of this path with the alternate color is necessary to avoid monochromatic triangles. Doing so would create a long monochromatic path in $T$. To avoid this problem, we allow some star edges to have a different color. For technical reasons, we encode the coloring of the stars in a 2-coloring of the vertices of $G$. The color of the center vertex is the color assigned to the star, while the color of a leaf shows how the edge is colored. 
Note that vertex-colorings in this section are not necessarily proper.

\begin{definition}\label{def-rebellious}
Let $c$ be a vertex $2$-coloring of an outing $G=(S,T)$. We say that an edge $\overrightarrow{uv} \in E(S)$ is \emph{rebellious} if $c(u)\ne c(v)$. We also call $v\in V(G)$ \emph{rebellious} if it is the head of a rebellious edge.
\end{definition}

We are mainly concerned with colorings where the rebellious vertices behave nicely with respect to $T$ in the following sense.

\begin{definition}\label{def-tame}
Let $c$ be a vertex $2$-coloring of an outing $G=(S,T)$. We say that $c$ is \emph{tame} if for every edge $\overrightarrow{uv}\in E(T)$ where $v$ is rebellious, we have $c(u)\ne c(v)$ and $u$ is not rebellious.
\end{definition}

In particular, it follows that if the 2-coloring is tame then two rebellious vertices are never joined by an edge in $T$. Notice that in a tame 2-coloring it is possible that all edges of a star are rebellious. 

Given a 2-vertex-coloring of an outing $G=(S,T)$, we now define a 2-edge-coloring of $G$ as follows.

\begin{definition}\label{def-ext}
Let $c:V(G)\rightarrow\{1,2\}$ be a vertex $2$-coloring of an outing $G=(S,T)$. The \emph{extension} of $c$, denoted by Ext($c$), is the $2$-edge-coloring $c':E(G)\rightarrow\{1,2\}$ where:

\begin{enumerate}
\item For all edges $\overrightarrow{uv}\in E(S)$, we have $c'(\overrightarrow{uv}) = c(v)$.
\item For all edges $\overrightarrow{uv}\in E(T)$, we have 
\[
c'(\overrightarrow{uv}) =
\begin{cases}
	c(v) & \text{ if } v\in C(S),\, c(u)=c(v) \text{ and } u \text{ is not rebellious,}\\
	3-c(v) & \text{ otherwise}.
\end{cases}
\]
\end{enumerate}
\end{definition}

Notice that in the Ext($c$)-coloring of $G$, every vertex $v\in V(G)$ has indegree at most 1 in each color. This implies that each monochromatic cycle is directed and each monochromatic path is the union of two directed paths.

\begin{definition}\label{def-center}
The \emph{center graph} $\text{Center}(G)$ of an outing $G=(S,T)$ is a directed graph whose vertex set is $C(S)$ and for every $u, v\in C(S)$ with $u\ne v$, there is an edge $\overrightarrow{uv}$ if $\overrightarrow{uv}\in E(T)$ or if there exists a vertex $w\in L(S)$ such that $\overrightarrow{uw}\in E(S)$ and $\overrightarrow{wv}\in E(T)$.
\end{definition}

Each vertex in $\text{Center}(G)$ has indegree at most 1. In particular, each cycle in $\text{Center}(G)$ is directed and each connected component contains at most one cycle.
Given a coloring of the vertices of $G$, this also corresponds to a coloring of Center($G$) in a natural way.

\begin{definition}\label{def-res}
Let $c$ be a vertex $2$-coloring of an outing $G$.
The \emph{center restriction} of $c$, denoted by Res($c$), is the vertex 2-coloring of Center($G$) defined by coloring each vertex $v\in V(\text{Center}(G))$ with color $c(v)$.
\end{definition}

Our first lemma characterizes monochromatic paths in Ext($c$) where the two endvertices of the path are in $C(S)$ and its interior vertices are in $L(S)$. Note that we phrase the lemma only for monochromatic paths in color 1, but the analogous statement holds also for paths in color 2.

\begin{lemma}\label{lemma-path}
Let $c:V(G)\rightarrow\{1,2\}$ be a tame vertex 2-coloring of an outing $G=(S,T)$. Let $P=v_0v_1\ldots v_k$ be a dipath in $G$ whose edges are colored 1 in Ext($c$). Suppose $v_0,v_k\in C(S)$ and $c(v_i)\in L(S)$ for $i\in \{1,\ldots ,k-1\}$. 

If $c(v_0)=c(v_k)$, then $k\leq 2$ and $\overrightarrow{v_0v_k}\in E(\text{Center}(G))$.

If $c(v_0)\neq c(v_k)$, then $k\leq 3$ and $c(v_0)=1$, $c(v_k)=2$.
\end{lemma}

\begin{proof}
First note that $\overrightarrow{v_{i}v_{i+1}}\in E(T)$ for $i\in\{1,\ldots ,k-1\}$ since $v_i\in L(S)$ for all such $i$. Now let us suppose $c(v_0)=c(v_k)$ and $k\geq 3$. Since $\overrightarrow{v_{k-2}v_{k-1}}\in E(T)$ and $v_{k-1}\in L(S)$, it follows from the definition of Ext($c$) that the color of $\overrightarrow{v_{k-2}v_{k-1}}$ in Ext($c$) (which is 1) equals $3-c(v_{k-1})$ and hence $c(v_{k-1})=2$.
If $c(v_k)=1$, then the edge $\overrightarrow{v_{k-1}v_k}$ would be colored 2 by the definition of Ext($c$), a contradiction. Thus $c(v_k)=2$ and $v_{k-1}$ is rebellious. Since $c$ is tame, we have $c(v_{k-2})=1$ and $v_{k-2}$ is not rebellious. 
By the definition of Ext($c$), it follows that $\overrightarrow{v_{k-3}v_{k-2}}\in E(S)$. Thus, $v_{k-3}\in C(S)$ and $k=3$. Since $c(v_0)=c(v_k)=2$ and $c(v_1)=1$, we have that $v_1$ is rebellious, a contradiction since $v_{k-2}$ is not rebellious.

Notice that $c(v_0)=c(v_k)$ and $k\leq 2$ implies $\overrightarrow{v_0v_k}\in E(\text{Center}(G))$ unless $k=2$ and $\overrightarrow{v_0v_1}, \overrightarrow{v_1v_2}\in E(T)$. As before, this case implies $c(v_1)=2$, $c(v_2)=2$ and $v_1$ is not rebellious. Since $c$ is tame, it follows that $c(v_0)=1$, contradicting $c(v_0)=c(v_2)$.

Next suppose $c(v_0)=2$ and $c(v_k)=1$. By the definition of Ext($c$), we have $c(v_{k-1})=1$ and $v_{k-1}$ is not rebellious. It follows that $k\geq 2$. Once again, it follows that $\overrightarrow{v_{k-2}v_{k-1}}\in E(S)$. Thus, $v_{k-2}\in C(S)$ and $k=2$. Now $c(v_0)=2$ and $c(v_1)=1$, so $v_1$ is rebellious, a contradiction.

Finally, suppose $c(v_0)=1$, $c(v_k)=2$ and $k>3$. Since the edges $\overrightarrow{v_{k-3}v_{k-2}}$ and $\overrightarrow{v_{k-2}v_{k-1}}$ are in $E(T)$ and colored 1, we have $c(v_{k-2})=c(v_{k-1})=2$. Now $v_{k-1}$ is not rebellious since $c$ is tame, so the edge $\overrightarrow{v_{k-1}v_k}$ received color 2 in Ext($c$), a contradiction.
\end{proof}

Let $c$ be a vertex-coloring (resp. edge-coloring) of a directed graph $G$. We say that $c$ is \textit{acyclic} if there exists no directed cycle in $G$ in which all vertices (resp. edges) have the same color. We want to find a vertex 2-coloring $c$ of $G$ such that Ext($c$) is acyclic. The next lemma shows that this goal is achieved whenever $c$ is tame and the restriction of $c$ is acyclic.

\begin{lemma}~\label{lemma-acyclic}
Let $c:V(G)\rightarrow\{1,2\}$ be a tame vertex 2-coloring of an outing $G=(S,T)$. If Res($c$) is acyclic, then also Ext($c$) is acyclic. 
\end{lemma}
\begin{proof}
Suppose not. Let $C$ be a monochromatic cycle in Ext($c$), say in color 1. 
We set $C_C=V(C)\cap C(S)$ and $C_L=V(C)\cap L(S)$. Notice that both $C_C$ and $C_L$ are non-empty since $C$ must contain an edge of $S$ as $T$ is a tree. Let $v_0\in C_C$ and label the remaining vertices in $C_C$ by $v_1,\ldots ,v_n$ as they appear in $C$ starting from $v_0$. 

First let us suppose that not all vertices in $C_C$ are colored the same. 
Then there exists an $i\in\{0,\ldots n\}$ such that $c(v_i)=2$ and $c(v_{i+1})=1$ (indices are considered modulo $n+1$). Now the directed path from $v_i$ to $v_{i+1}$ on $C$ contradicts Lemma~\ref{lemma-path}.
We may thus assume that all vertices in $C_C$ received the same color. By Lemma~\ref{lemma-path}, the paths between $v_i$ and $v_{i+1}$ on $C$ correspond to edges in Center($G$). Thus, the vertices $v_0,\ldots ,v_n$ correspond to a monochromatic cycle in Center($G$), contradicting that Res($c$) is acyclic.
\end{proof}

Now we give an upper bound for the length of a monochromatic dipath in Ext($c$).

\begin{lemma}\label{MaxLength} 
Let $c:V(G)\rightarrow\{1,2\}$ be a tame vertex 2-coloring of an outing $G=(S,T)$ for which Res($c$) is acyclic. Let $d_T$ be the length of a longest vertex-monochromatic dipath in $T$ whose vertices are all in $L(S)$. For $i\in\{1,2\}$, let $d_i$ be the length of a longest monochromatic dipath in Center($G$) whose vertices are colored $i$ in Res($c$). If $P$ is a monochromatic dipath in the Ext($c$)-coloring of $G$, then the length of $P$ is at most
$d_T+2(d_1+d_2)+6\,.$
\end{lemma}

\begin{proof}
By Lemma~\ref{lemma-acyclic}, we know that Ext($c$) is acyclic. We may assume that the edges of $P$ are all colored 1.  
Let $v_0,v_1\ldots, v_n$ denote the vertices in $V(P)\cap C(S)$, labelled in the order they appear on $P$.
By Lemma~\ref{lemma-path} there exists no $i\in \{0,\ldots ,n-1\}$ with $c(v_{i})=2$ and $c(v_{i+1})=1$. Thus, there exists $k\in\{0,\ldots ,n+1\}$ such that $c(v_i)=1$ if and only if $i<k$. Notice that by Lemma~\ref{lemma-path}, the vertices $v_0v_1\ldots v_{k-1}$ correspond to a monochromatic path of color 1 and length $k-1$ in Center($G$), while the vertices $v_kv_{k+1}\ldots v_{n}$ correspond to a monochromatic path of color 2 and length $n-k$. By definition of $d_1$ and $d_2$ we have $k-1\leq d_1$ and $n-k\leq d_2$.
By Lemma~\ref{lemma-path}, there are at most 3 edges on $P$ between $v_{k-1}$ and $v_k$, and at most 2 edges between $v_{i-1}$ and $v_i$ for every $i\in\{1,\ldots ,n\}\setminus\{k\}$. Thus, the number of edges on $P$ between $v_0$ and $v_n$ is at most $2(k-1) + 3 + 2(n-k) \leq 2(d_1+d_2)+3$. 

Let $w_0, \ldots ,w_{n'}$ denote the vertices encountered on $P$ after $v_n$. Then $w_i\in L(S)$ for $i\in \{0,\ldots ,n'\}$ and $\overrightarrow{w_iw_{i+1}}\in E(T)$ for $i\in \{0,\ldots ,n'-1\}$. Since the edges of $P$ are all colored 1, we have $c(w_i)=2$ for $i\in \{1,\ldots ,n'\}$. Thus $n'-1\leq d_T$, and there are at most $d_T+2$ edges on $P$ after $v_n$.

Suppose there are at least 3 edges on $P$ before $v_0$, say $\overrightarrow{u_0u_1}$, $\overrightarrow{u_1u_2}$, and $\overrightarrow{u_2v_0}$. Then all these three edges must be in $T$ and $c(u_1)=c(u_2)=2$. Thus, $u_2$ is not rebellious, and no matter what the the color of $v_0$ is, the edge $\overrightarrow{u_2v_0}$ is colored 2 in Ext($c$), a contradiction. Suppose there are two edges $\overrightarrow{u_1u_2}$ and $\overrightarrow{u_2v_0}$ before $v_0$. Then $c(u_2)=2$ and since the edge $\overrightarrow{u_2v_0}$ is colored 1, it follows that $c(v_0)=2$. In this case there are at most $2d_2$ edges between $v_0$ and $v_n$, so the length of $P$ is at most $2+2d_2+d_T+2<d_T+2(d_1+d_2)+6$. Finally, suppose there is at most one edge preceding $v_0$ in $P$. Then the length of $P$ is at most $1+2(d_1+d_2)+3 + d_T+2= d_T+2(d_1+d_2)+6$. 
\end{proof}

Finally, all that is left to show is that there exists a vertex 2-coloring of $G$ satisfying the conditions of Lemma~\ref{MaxLength}. 

\begin{lemma}\label{ExTame}
Let $G=(S,T)$ be an outing. There exists a tame vertex 2-coloring $c$ of $G$ such that
color class 1 of Res($c$) forms an independent set in Center($G$),
color class 2 of Res($c$) induces no directed path of length 2 in Center($G$), and
there is no vertex-monochromatic dipath of length 2 in $T$ whose vertices are all in $L(S)$.
\end{lemma}

\begin{proof}
We start by coloring the vertices in $C(S)$. If a component of Center($G$) is bipartite, then we choose a proper 2-coloring of its vertices. If a component is not bipartite, then it contains precisely one cycle and this cycle has odd length. In this case we delete an edge $uv$ of that cycle and properly 2-color the resulting tree so that $c(u)=2$.  Now the two color classes of Res($c$) are as desired. 

We now extend this coloring to the vertices in $L(S)$. If the root of $T$ is in $L(S)$, color it arbitrarily. Let $v$ be a vertex at distance $i$ from the root in $T$ and suppose all vertices at distance $i-1$ from the root are already colored. Let $u$ be the parent of $v$ in $T$ and let $w$ be such that $\overrightarrow{wv} \in E(S)$. We set $c(v)=3-c(u)$ unless $u$ is rebellious and $c(u)=c(w)$, in which case we set $c(v)=c(u)$. Notice that if $c(v)=c(u)$, then $v$ is not rebellious. Thus if $c(v)\neq c(u)$ and $v$ is rebellious, then $c(u)=c(w)$; in which case $u$ is not rebellious given how we set the color of $v$. This implies that the resulting coloring $c$ is tame. Furthermore, if $\overrightarrow{uv}$ is an edge with $u,v\in L(S)$ and $c(u)=c(v)$, then $u$ is rebellious while $v$ is not rebellious. It follows immediately that there are no vertex-monochromatic dipaths of length 2 in $T$ whose vertices are in $L(S)$.
\end{proof}

Now Theorem~\ref{TreeStar} follows easily.

\begin{proof}[Proof of Theorem~\ref{TreeStar}]
Let $G$ be the union of a forest and a star forest. Now let $G'=(S,T)$ be an outing such that the underlying undirected graph of $G'$ contains $G$ as a subgraph.
Let $c$ be a tame vertex 2-coloring of $G'$ as given by Lemma~\ref{ExTame}. Let $H'$ be a monochromatic connected subgraph of $G'$ and let $H$ be the underlying undirected graph of $H'$. 

Suppose $H$ contains a cycle $C$. Since the indegree of every vertex in $H'$ is at most one, the cycle $C$ is directed in $H'$. By Lemma~\ref{lemma-acyclic}, there are no monochromatic directed cycles in Ext($c$), a contradiction. So we may assume that $H$ is a tree.
By Lemma~\ref{MaxLength}, the length of a monochromatic dipath in Ext($c$) is at most $1+2\cdot (0+1)+6=9$. Thus, every dipath in $H'$ has length at most 9. Since the indegree of every vertex in $H'$ is at most one, every path in $H$ is the union of at most two dipaths in $H'$. Thus, the diameter of $H$ is at most 18. Hence, $c$ induces a 2-edge-coloring of $G$ in which every connected monochromatic subgraph is a tree with diameter at most 18.
\end{proof}

\section{Planar graphs and $\varepsilon$-thin spanning trees}

All graphs in this section are finite and planar. We denote the dual of a planar graph $G$ by $G^*$.
Given a graph $G$ and a set of vertices $A\subseteq V(G)$, we denote by $\sigma _G(A)$ the set of edges of the form $\{ab \in E(G):a\in A,b\notin A\}$. We call $\sigma_G(A)$ the \emph{boundary} of $A$ in $G$.

\begin{definition}
Let $\varepsilon$ be a real number with $0<\varepsilon <1$. We say a spanning subgraph $H$ of a graph $G$ is \textit{$\varepsilon$-thin} if for every $A\subseteq V(G)$ we have $|\sigma_H(A)| \leq \varepsilon |\sigma_G (A)|$.
\end{definition}

Of particular interest is the existence of $\varepsilon$-thin spanning trees. Goddyn~\cite{Goddyn} conjectured that for every $\varepsilon$ with $0<\varepsilon <1$ there exists a number $f(\varepsilon )$ such that every $f(\varepsilon )$-edge-connected graph contains an $\varepsilon$-thin spanning tree. This would imply the $(2+\varepsilon )$-flow conjecture by Goddyn and Seymour, which was recently proved by Thomassen~\cite{Thomassen}.

Thomassen observed that there exists no real number $\varepsilon$ with $0<\varepsilon < 1$ such that every 4-edge-connected planar graph contains an $\varepsilon$-thin spanning tree (personal communication). Here we give a short proof inspired by his argument.
\begin{theorem}
For every real number $\varepsilon$ with $0<\varepsilon < 1$ there exists a planar 4-edge-connected graph with no $\varepsilon$-thin spanning tree.
\end{theorem}
\begin{proof}
We fix $\varepsilon$ and set $k>\max \{\lceil \frac{3}{1-\varepsilon}\rceil ,1000\}$. Let $G$ be the cartesian product of a path of length $4k$ and a cycle of length $4k$. The graph $G$ is planar but not 4-edge-connected since there exist $8k$ vertices of degree 3 which lie on two faces each containing $4k$ vertices of degree 3. We add new vertices inside these faces and join each new vertex to 4 vertices of degree 3 so that the resulting graph is planar, 4-regular and 4-edge-connected. Moreover, it is easy to see that the resulting graph $G'$ has the property that every sufficiently large set of vertices has a large neighborhood. We leave the verification of the following statement to the reader: For every $A\subseteq V(G')$ with $k^2\leq |A|\leq |V(G')|-k^2$, we have $|\sigma_{G'}(A)|\geq k$. 

Suppose for a contradiction that $G'$ has an $\varepsilon$-thin spanning tree $T$. Since $T$ is $\varepsilon$-thin, the graph $G'-E(T)$ is connected. Let $T'$ be a spanning tree of $G'-E(T)$. Since $G'$ is 4-regular, we have $|E(G'-E(T)-E(T'))|=2n-2(n-1)=2$. Let $e$ be an edge of $T'$ such that $T'-e$ has two connected components $A$ and $B$ each having size at least $k^2$ (such an edge exists since the maximum degree of $T'$ is 4). Thus, $|\sigma_{G'}(A)|\geq k$, but only one of the edges in $\sigma_{G'}(A)$ is contained in $T'$. Since there exist only two edges in $G'$ outside of $T$ and $T'$, the proportion of $\sigma_{G'}(A)$ contained in $T$ is at least 
$$\frac{|\sigma_{G'}(A)|-3}{|\sigma_{G'}(A)|} \geq \frac{k-3}{k} = 1- \frac{3}{k} > \varepsilon \,,$$
contradicting $T$ being $\varepsilon$-thin.
\end{proof}

The following lemma shows that bounded diameter arboricity of planar graphs is related to the existence of $\varepsilon$-thin spanning trees. 

\begin{lemma}\label{lemma-dual}
If $G$ is a planar graph with $\Upsilon_d(G)= 2$, then $G^*$ contains two edge-disjoint $\frac{d}{d+1}$-thin spanning trees.
\end{lemma}
\begin{proof}
Since $\Upsilon_d(G)= 2$, we can edge-color $G$, say in colors 1 and 2, so that there are no monochromatic cycles and every monochromatic path has length at most $d$. By the usual bijection $E(G)\rightarrow E(G^*)$, this gives a 2-edge-coloring of $G^*$. Consider a set $A\subseteq V(G^*)$. The edges in $\sigma_{G^*}(A)$ correspond to an edge-disjoint union of cycles in $G$. Consider one such cycle $C$ in the union. Since there are no monochromatic cycles in $G$, both colors appear in $C$. Moreover, since every path of length at least $d+1$ contains an edge in color 1, at least $\frac{1}{d+1}|E(C)|$ edges of $C$ are colored 1. Thus, at most $\frac{d}{d+1}|\sigma_{G^*}(A)|$ edges of $\sigma_{G^*}(A)$ are colored 2. Since $\sigma_{G^*}(A)$ also contains at least $\frac{1}{d+1}|\sigma_{G^*}(A)|$ edges in color 2, the subgraph colored 2 is both spanning and $\frac{d}{d+1}$-thin. The same holds for the subgraph in color 1. Since subgraphs of $\varepsilon$-thin graphs are again $\varepsilon$-thin, we can choose one spanning tree of $G^*$ in each color to finish the proof.
\end{proof}

We should note that planar graphs of various girths have received much attention for star arboricity (their arboricity is at most 3 for all planar graphs, and at most 2 for triangle-free planar graphs by Euler's formula). Thus we wondered what the bounded diameter arboricity of planar graphs of various girths was. Upon studying the problem, we began to conjecture that planar graphs have bounded diameter arboricity at most 4; similarly, we conjectured that planar triangle-free graphs have bounded diameter arboricity at most 3. Indeed, this is what led us to Conjecture~\ref{PlusOneConj}. Theorem~\ref{TreeStar} has allowed us to prove these conjectures. To see that the bounded diameter arboricity of these classes is greater than the usual arboricity, we use the following lemma.

\begin{lemma}\label{lemma-lowerbound}
Let $\mathcal{G}\subseteq \mathcal{A}_k$ be a family of graphs and $c$ a natural number. If there exists a sequence of graphs $G_1,G_2,\ldots $ in $\mathcal{G}$ such that the diameter of $G_i$ is at least $i$ and $|E(G_i)|\geq k|V(G_i)|-c$ for all $i$, then $\Upsilon_{bd}(\mathcal{G}) \geq k+1$.
\end{lemma}

\begin{proof}
Suppose $\Upsilon_{bd}(\mathcal{G}) \leq k$, then there exists a natural number $d$ such that $\Upsilon_{d}(G)\leq k$ for all $G\in \mathcal{G}$. Consider the graph
$H=G_{cd+1}$. Let $\mathcal{F}=\{F_1,\ldots ,F_k\}$ be a decomposition of $H$ into $k$ forests in which each tree has diameter at most $d$. For $i\in \{1,\ldots ,k\}$, let $\mathcal{T}_i$ denote the connected components of $F_i$ (if a vertex of $H$ is not contained in $F_i$ then we include it in $\mathcal{T}_i$ as an isolated vertex). Now $\mathcal{T} = \bigcup_{i=1}^k \mathcal{T}_i$ is a collection of trees decomposing $H$, each having diameter at most $d$.
Notice that
$$k|V(H)|-c \leq |E(H)|= \sum_{i=1}^k |E(F_i)| = \sum_{i=1}^k |V(H)|-|\mathcal{T}_i| \leq k|V(H)| - |\mathcal{T}|\,,$$
so $|\mathcal{T}|\leq c$. Since the diameter of $H$ is at least $cd+1$, there exists a path $P$ of length at least $cd+1$ in $H$ such that $P$ is a shortest path between its endpoints. Since $P$ contains $cd+1$ edges and every edge is contained in a tree of $\mathcal{T}$, there exists a tree $T$ in $\mathcal{T}$ containing at least $d+1$ edges of $P$. However, since $P$ is a shortest path, this implies that the diameter of $T$ is greater than $d$, contradicting our choice of $\mathcal{F}$.
\end{proof}

For planar graphs of higher girth, we were led to conjecture that planar graphs of girth at least 5 have bounded diameter arboricity at most 2. We were only able to prove this for girth at least 6 and only then by using the result of Kim et al.~\cite{NDT2} that a planar graph of girth at least 6 can be decomposed into a forest and a matching.

\begin{thm}\label{thm-planar}
If we let $\P_g$ denote the class of planar graphs of girth at least $g$, then

\begin{itemize}
\item $\Upsilon_{bd}(\P_3) = 4$,
\item $\Upsilon_{bd}(\P_4) = 3$,
\item $\Upsilon_{bd}(\P_g) = 2$ for all $g\ge 6$.
\end{itemize}
\end{thm}
\begin{proof}
By Euler's formula $\Upsilon(\P_3)=3$ and hence by Theorem~\ref{23}, $\Upsilon_{bd}(\P_3)\le 4$. Since there exist planar triangulations of arbitrary diameter (and hence $|E(G)|=3|V(G)|-6$), it follows from Lemma~\ref{lemma-lowerbound} that $\Upsilon_{bd}(\P_3) = 4$. Similarly by Euler's formula $\Upsilon(\P_4)=2$. By Theorem~\ref{23}, $\Upsilon(\P_4)\le 3$. Since there exist triangle-free planar graphs of arbitrary diameter with $|E(G)|=2|V(G)|-4$, it follows from Lemma~\ref{lemma-lowerbound} that $\Upsilon(\P_4)=3$.

For $g\ge 5$, clearly $\Upsilon_{bd}(\P_g)\ge 2$. By Kim et al.~\cite{NDT2}, every planar graph of girth at least six can be decomposed into a forest and a matching. Thus by Theorem~\ref{TreeStar}, every planar graph of girth at least six can be decomposed into two forests whose components have diameter at most 18. Hence $\Upsilon_{bd}(\P_6)=2$ and $\Upsilon_{bd}(\P_g)=2$ for all $g\ge 6$.
\end{proof}

Notice that Lemma~\ref{lemma-dual} still holds when $G^*$ has multiple edges. Thus we have the following corollary.

\begin{corollary}
Every 6-edge-connected planar (multi)graph contains two edge-disjoint $\frac{18}{19}$-thin spanning trees. 
\end{corollary}
\begin{proof}
Let $G$ be a $6$-edge-connected planar (multi)graph. As $G$ is $6$-edge-connected, it follows that the dual $G^*$ of $G$ is a simple planar graph of girth at least six. As in Theorem~\ref{thm-planar}, we find that $\Upsilon_{18}(G^*)=2$. By Lemma~\ref{lemma-dual}, $(G^*)^*=G$ contains two edge-disjoint $\frac{18}{19}$-thin spanning trees.
\end{proof}

\section{Open problems}

As we have seen, bounded diameter arboricity differs from star arboricity for the class of planar graphs (5 instead of 4). The only missing case in Theorem~\ref{thm-planar} is $g=5$. Clearly, $2 \leq \Upsilon_{bd}(\P_5) \leq \Upsilon_{bd}(\P_4) = 3$. We conjecture that the following holds.

\begin{conjecture}
$\Upsilon_{bd}(\P_5)=2$.
\end{conjecture}

This conjecture would be implied by Theorem~\ref{TreeStar} if the answer to the following question is affirmative.

\begin{question}
Is every planar graph of girth 5 the union of a forest and a star forest?
\end{question}

As before, a positive answer to this question would also imply that every 5-edge-connected planar graph contains two disjoint $\frac{18}{19}$-thin spanning trees. It is not even known whether there exists an $\varepsilon$ such that every 5-edge-connected planar graph contains an $\varepsilon$-thin spanning tree.
 
For the general problem, Theorem~\ref{TreeStar} suggests a strategy for proving Conjecture~\ref{PlusOneConj}. We conjecture the following generalization of Theorem~\ref{TreeStar} holds.

\begin{conjecture}\label{TreeBush}
For all natural numbers $d\ge 1$, there exists a natural number $f(d)$ such that the following holds: If $G$ is the union of a forest and a second forest whose components have diameter at most $d$, then $G$ can be partitioned into two forests each of whose components have diameter at most $f(d)$.
\end{conjecture}

Thus our main result confirms this conjecture when $d\le 2$ with $f(2)\le 18$.

One may also wonder if there is a stronger variant of Conjecture~\ref{PlusOneConj}. This could be possible if we allow the arboricity to be fractional. The \emph{fractional arboricity} $\Upsilon_f(G)$ is defined as $\max_{H\subseteq G} \frac{|E(H)|}{|V(H)|-1}$. Note that $\lceil \Upsilon_f(G) \rceil = \Upsilon(G)$ by Nash-Williams' result. A major open question is whether the structure of the forests can be restricted when the fractional arboricity is strictly smaller (asymptotically) than the arboricity. In particular, Montassier et al.~\cite{NDT} formulated the Nine Dragon Tree Conjecture as follows.

\begin{conjecture}[Nine Dragon Tree Conjecture]\label{NDT}
Let $G$ be a graph and $k,d$ natural numbers with $k,d\geq 1$. If $\Upsilon_f(G)\le k + \frac{d}{k+d+1}$, then $G$ can be decomposed into $k+1$ forests at least one of which has maximum degree $d$.
\end{conjecture}

They proved Conjecture~\ref{NDT} for $k=1$ and $d\le 2$. Kim et al.~\cite{NDT2} proved the conjecture for $k=1$ and $d\le 6$.  The \emph{Strong Nine Dragon Tree Conjecture} states that for such graphs at least one of the forests in the decomposition has components of size at most $d$ (and hence diameter at most $d$ as well). In light of Conjecture~\ref{TreeBush} and the Strong Nine Dragon Tree Conjecture, we also make the following strong conjecture.

\begin{conjecture}\label{FracBush}
For every natural number $k$ and real number $\varepsilon > 0$, there exists $d(k,\epsilon)$ such that the following holds: if $\Upsilon_f(G) \le k-\epsilon$ for a graph $G$, then $\Upsilon_{d(k,\epsilon)}(G)\le k$.
\end{conjecture}

\end{document}